\documentclass[a4paper,11pt]{amsart}
\usepackage{amssymb}
\usepackage{amsmath}
\usepackage{amscd}
\usepackage[all]{xy}

\usepackage{comment}

\usepackage{verbatim}
\usepackage{url}

\DeclareSymbolFont{rsfs}{U}{rsfs}{m}{n}
\DeclareSymbolFontAlphabet{\mathcal}{rsfs}

\DeclareFontEncoding{OT2}{}{} 
\DeclareTextFontCommand{\textcyr}{\fontencoding{OT2}
    \fontfamily{wncyr}\fontseries{m}\fontshape{n}\selectfont}

\newcommand{\Ker}{\operatorname{ker}}

\theoremstyle{plain}

\newtheorem{theorem}{Theorem}[section]
\newtheorem{maintheorem}[theorem]{Main Theorem}
\newtheorem{proposition}[theorem]{Proposition}
\newtheorem{lemma}[theorem]{Lemma}
\newtheorem{corollary}[theorem]{Corollary}

\theoremstyle{definition}
\newtheorem{subsec}[theorem]{}
\newtheorem{remark}[theorem]{Remark}

\newtheorem{definition}[theorem]{Definition}
\newtheorem{example}[theorem]{Example}

\newcommand{\Zz}{{\mathbb{Z}}}

\newcommand{\Rr}{{\mathbb{R}}}

\newcommand{\Cc}{{\mathbb{C}}}
\newcommand{\Nn}{{\mathbb{N}}}

\def\Z{\Zz}

\def\R{\Rr}
\def\C{\Cc}
\def\N{\Nn}

\newcommand{\GL}{{\operatorname{GL}}}
\newcommand{\SL}{{\operatorname{SL}}}
\newcommand{\Sp}{{\operatorname{Sp}}}
\newcommand{\PSp}{{\operatorname{PSp}}}

\newcommand{\PGL}{{\operatorname{PGL}}}
\newcommand{\SU}{{\rm SU}}
\newcommand{\PSU}{{\rm PSU}}

\newcommand{\Nm}{{\rm N}}

\newcommand{\Gal}{{\operatorname{Gal}}}

\newcommand{\ad}{^{\rm ad}}
\renewcommand{\sc}{^{\rm sc}}
\newcommand{\sss}{^{\rm ss}}

\newcommand{\X}{{\mathbb{X}}}

\newcommand{\G}{{\mathbb G}}
\newcommand{\Gm}{{\G_m}}

\renewcommand{\ggg}{{\mathfrak{g}}}

\newcommand{\reg}{{\rm reg}}

\newcommand{\ob}{\operatorname{ob}}

\newcommand{\Rad}{\operatorname{R}}
\newcommand{\isoto}{\overset{\sim}{\to}}

\newcommand{\into}{\hookrightarrow}
\newcommand{\onto}{\twoheadrightarrow}

\newcommand{\labelto}[1]{\xrightarrow{\makebox[1.5em]{\scriptsize ${#1}$}}}

\newcommand{\Ext}{{\rm Ext}}

\newcommand{\kbar}{{\overline{k}}}

\def\Kbar{K_s}

\def\charact{\operatorname{char}}

\title[Toric-friendly  groups]{ Toric-friendly  groups}
\author{Mikhail Borovoi}
\address{Borovoi: Raymond and Beverly Sackler School of Mathematical Sciences,
Tel Aviv University, 69978 Tel Aviv, Israel}
\email{borovoi@post.tau.ac.il}
\thanks{Borovoi was partially supported
by the Hermann Minkowski Center for Geometry}
\author{Zinovy Reichstein}
\address{Reichstein: Department of Mathematics, The University of British
Columbia, 1984 Mathematics Road, Vancouver, B.C., Canada V6T 1Z2}
\email{reichst@math.ubc.ca}
\thanks{Reichstein was partially supported by NSERC Discovery
and Accelerator Supplement grants}
\subjclass[2000]{Primary: 20G15; Secondary: 14G05, 20G10}
\keywords{Toric-friendly groups,   linear algebraic groups, semisimple groups, maximal tori,  rational points,
 elementary obstruction}

\newcommand{\bbG}{{\mathbb{G}}}

\def\Z{\Zz}
\def\bbZ{\Zz}

\def\R{\Rr}
\def\C{\Cc}
\def\N{\Nn}
\def\Ru{{R_{\rm u}}}
\def\simeq{\cong}

\begin{document}

\begin{abstract}
Let $G$ be a connected linear algebraic group over a field $k$.
We say that $G$ is  toric-friendly  if for any field extension
$K/k$ and any maximal $K$-torus $T$ in $G$
the group $G(K)$ acts transitively on $(G/T)(K)$.
Our main result is a classification of semisimple
(and under certain assumptions on $k$, of connected)
 toric-friendly  groups.
\end{abstract}

\maketitle

\section{Introduction}
Let $k$ be a field
and $X$ be a homogeneous space of a connected linear algebraic group $G$
defined over $k$.
The first question one usually asks about $X$
is whether or not it has a $k$-point.
If the answer is ``yes", then one
often wants to know whether or not the set $X(k)$ of $k$-points of $X$
forms a single orbit under the group $G(k)$.

In this paper we shall focus on the case where
the geometric stabilizers for the $G$-action on $X$ are maximal
tori of $G_\kbar:=G\times_k\kbar$
(here $\kbar$ stands for a fixed algebraic closure of $k$).
Such homogeneous spaces arise, in particular, in the study
of the adjoint action of a connected reductive group $G$ on
its Lie algebra or of the conjugation action of $G$ on
itself, see~\cite{ckpr}.
It is shown in~\cite[Corollary 4.6]{ckpr}
(see also \cite[Lemma 2.1]{Kottwitz})
that every homogeneous space $X$ of this type has a $k$-point,
assuming that $G$ is split and $\charact(k)=0$.
Therefore it is natural to ask if this point is unique
up to translations by $G(k)$.

\begin{definition}\label{def:toric-friendly}
Let $k$ be a field.
We say that a connected linear $k$-group $G$ is \emph{toric-friendly}
if for every field extension $K/k$ the following condition is satisfied:

\setlength{\hangindent}{30pt}
(*) For every maximal $K$-torus $T$ of $G_K:=G\times_k K$
the group $G(K)$ has only one orbit in $(G_K/T)(K)$, or, equivalently,
the natural map $\pi\colon G(K) \to (G_K/T)(K)$ is surjective.
\end{definition}

Examining the cohomology exact sequence associated to
the $K$-subgroup $T$ of $G_K$
(cf. \cite[I.5.4, Proposition 36]{Serre-CG}),
we see that
$G$ is toric-friendly if and only if $\ker[H^1(K,T)\to H^1(K,G)]=1$
for every field extension  $K/k$ and every maximal $K$-torus $T$ of $G_K$.
\medskip

Observe that $G$ is toric-friendly if and only if 
condition (*) of Definition \ref{def:toric-friendly}
is satisfied for all \emph{finitely generated} extensions $K/k$.
\medskip

We are interested in classifying   toric-friendly  groups.
In Section~\ref{sect.prel} we partially reduce this problem
to the case where the group is semisimple. The rest of this paper
will be devoted to proving the following classification theorem for
semisimple   toric-friendly  groups.

\begin{maintheorem} \label{thm:non-split-intro}
Let $k$ be a field.
A connected semisimple $k$-group $G$
is  toric-friendly  if and only if $G$ is isomorphic to a direct product
$\prod_i R_{F_i/k}G'_i$, where each $F_i$ is a finite separable extension of $k$
and each $G'_i$ is an inner form of $\PGL_{n_i,F_i}$
for some integer $n_i$.
\end{maintheorem}

\bigskip
\noindent\emph{Acknowledgements.}
We are grateful to Jean-Louis Colliot-Th\'el\`ene, an editor of ANT,
for helpful comments and suggestions. 
In particular, he contributed Lemma~\ref{lem:ElemObstr}
and the idea of Proposition \ref{prop.nontrivial}, which
simplified our earlier arguments.  We  thank the anonymous referee
for a quick and thorough review and an anonymous editor of ANT for helpful comments.
We also thank Brian Conrad, Philippe Gille,
Boris Kunyavski\u\i\, and James S.~Milne for stimulating
discussions.

\tableofcontents

\noindent\emph{Notation.}

Unless otherwise specified, $k$ will denote an arbitrary field.
For any field $K$ we denote by $\Kbar$  a separable closure of $K$.

By a $k$-group we mean an affine algebraic group scheme over $k$,
not necessarily smooth or connected.
However, by a reductive $k$-group (resp.~a semisimple $k$-group)
we mean a smooth, connected, reductive $k$-group
(resp.~a smooth, connected, semisimple $k$-group).

Let $S$ be a $k$-group.
We  denote by $H^i(k,S)$ the $i$-th flat cohomology
set for $i=0,1$, cf.~\cite[17.6]{Waterhouse}.
If $S$ is abelian, we  denote by $H^i(k,S)$ the $i$-th flat cohomology
group for $i\ge 0$, cf.~\cite[Appendix B]{BFT}.
There are exact sequences for flat cohomology
similar to those for Galois cohomology,
see~\cite[18.1]{Waterhouse} and~\cite[Appendix B]{BFT}.
When $S$ is smooth, the flat cohomology
$H^i(k,S)$ can be identified with Galois cohomology.

\section{First reductions}
\label{sect.prel}

\begin{lemma} \label{lem:UGG'}
 Let
$
1\to U\to G\labelto{\varphi} G' \to 1
$
be an exact sequence of smooth connected $k$-groups, where $U$ is unipotent.
We assume that $U$ is $k$-split, i.e. has a composition series over $k$
whose successive quotients are isomorphic to $\G_{{\rm a},k}$.
Then $G$ is toric-friendly if and only if $G'$ is toric-friendly.
\end{lemma}

\begin{proof}
Choose a field extension $K/k$ and a maximal $K$-torus $T\subset G_K$.
Set $T'=\varphi(T)\subset G'_K$, then $T'$ is a maximal torus of $G'_K$.
The  map $\varphi^T\colon T\to T'$ is an isomorphism,
because $T \cap U_K = 1$ (as $U_K$ is unipotent).
Conversely, let us start from a maximal torus $T'$  of $G'_K$.
Let $H=\varphi^{-1}(T')\subset G_K$ be the preimage of $T'$, then $H$ is smooth and connected, so
any maximal torus $T$ of $H$ maps isomorphically onto $T'$ and therefore it is maximal in  $G_K$.

Now we have  a commutative diagram
\begin{equation*}
\xymatrix{
H^1(K,T)  \ar@{->}[r] \ar[d]_{\varphi^T_*} & H^1(K,G)   \ar@{->}[d]^{\varphi_*}\\
H^1(K,T') \ar@{->}[r] & H^1(K,G')
}
\end{equation*}
Since  $\varphi^T\colon T\to T'$  is an isomorphism of tori,
the left vertical arrow $\varphi^T_*$ is an isomorphism of abelian groups.
On the other hand, by \cite[Lemme 1.13]{Sansuc} the right vertical arrow $\varphi_*$  is a bijective map.
We see that the the top horizontal arrow in the  diagram is injective 
if and only if the bottom  horizontal arrow  is injective,
which proves the lemma.
\end{proof}

Let $k$ be a perfect field and $G$ be a connected $k$-group.
Recall that over a perfect field  the unipotent radical of $G$ makes sense,
i.e., the "geometric" unipotent radical over an algebraic closure is defined over $k$, by Galois descent.
We denote the unipotent radical of $G$ by $\Ru(G)$.

\begin{corollary}\label{prop:non-reductive}
Let $k$ be a perfect field, $G$ be a  connected
$k$-group, and $\Ru(G)$ be its unipotent radical.
Then $G$ is  toric-friendly  if and only if the associated
reductive $k$-group $G/\Ru(G)$ is  toric-friendly.
\end{corollary}

\begin{proof}
 Since $k$ is perfect, the smooth connected unipotent $k$-group $\Ru(G)$ is $k$-split,
cf. \cite[Theorem 15.4]{Borel}, and the corollary follows from Lemma \ref{lem:UGG'}.
\end{proof}

Let $k$ be a field.
We recall that a  $k$-group $G$ is called
{\em special} if $H^1(K, G) = 1 $ for every field extension $K/k$.
This notion was introduced by J.-P.~Serre in~\cite{serre}.
Semisimple special groups over an algebraically closed field
were classified by A.~Grothendieck~\cite{Grothendieck};
we shall use his classification later on.

Recall that a $k$-torus $T$ is called quasi-trivial,
if its character group $\X(T)$ is a permutation Galois module.
Split tori and, more general, quasi-trivial tori are special.

\begin{proposition} \label{prop:CGG'}
Let
$
1\to C\to G\labelto{\varphi} G' \to 1
$
be an exact sequence of $k$-groups, where
$G$ and $G'$ are  reductive,
and $C \subset G$ is central, hence
of multiplicative type (not necessarily
connected or smooth).

\smallskip
(a) If $G$ is  toric-friendly  then so is $G'$.

\smallskip
(b) If $C$ is a special $k$-torus
then $G$ is  toric-friendly  if and only if $G'$ is  toric-friendly.
\end{proposition}

\begin{proof}
Let $K/k$ be a field extension.
The map $T\mapsto T':=\varphi(T)$
is a bijection between the set of maximal $K$-tori $T\subset G_K$
and the set of maximal $K$-tori $T'\subset G'_K$
(the inverse map is $T'\mapsto T:=\varphi^{-1}(T')$).
For such $T$ and $T'=\varphi(T)$
we have  commutative diagrams
\begin{equation*}\label{eq:diagram-CGG'}
\xymatrix{
G_K   \ar@{->}[r]^{\varphi} \ar[d]_{\pi} & G_K'   \ar@{->}[d]^{\pi'}
&G(K )  \ar@{->}[r]^{\varphi} \ar[d]_{\pi} & G'(K)   \ar@{->}[d]^{\pi'}\\
G_K/T \ar@{->}[r]^{\varphi_*}_{\simeq} & G_K'/T'
&(G_K/T)(K) \ar@{->}[r]^{\varphi_*}_{\simeq} & (G_K'/T')(K)
}
\end{equation*}
where $\varphi_*\colon G_K/T\isoto G'_K/T'$
is an isomorphism of $K$-varieties,
and the induced map
on $K$-points $\varphi_*\colon (G_K/T)(K)\to (G'_K/T')(K)$
is a bijection.
Now, if $G$ is  toric-friendly, then the map
$\pi\colon G(K)\to(G_K/T)(K)$ is surjective,
and we see from the right-hand diagram that then the map
$\pi'\colon G'(K)\to(G'_K/T')(K)$ is surjective as well.
This shows that $G'$ is  toric-friendly, thus proving (a).

To prove (b), assume that $G'$ is  toric-friendly  and $C$ is a special $k$-torus.
Then the map $\pi'\colon G'(K)\to(G'_K/T')(K)$ is surjective
(because $G'$ is  toric-friendly)
and the map $\varphi\colon G(K)\to G'(K)$ is surjective
(because $C$ is special).
We see from the right-hand diagram
that the map $\pi\colon G(K)\to(G_K/T)(K)$ is surjective as well.
Hence $G$ is  toric-friendly.
\end{proof}

We record the following immediate corollary of Proposition~\ref{prop:CGG'}(b).

\begin{corollary} \label{cor.modR(G)}
Let $G$ be a reductive $k$-group.
Suppose that the radical $\Rad(G)$ is a special $k$-torus (in particular,
this condition is satisfied if $\Rad(G)$ is a quasi-trivial $k$-torus).
Then $G$ is  toric-friendly  if and only if the semisimple group $G/\Rad(G)$ is  toric-friendly.
\qed
\end{corollary}

The following corollary follows from Corollary \ref{prop:non-reductive}
and Corollary \ref{cor.modR(G)}.

\begin{corollary} \label{cor:non-reductive-modR(G)}
Let $k$ be a perfect field.
Let $G$ be a connected  $k$-group containing a split maximal torus.
Then $G$ is  toric-friendly  if and only if the semisimple group $G/\Rad(G)$ is  toric-friendly.
\qed
\end{corollary}

Corollary~\ref{cor:non-reductive-modR(G)}
partially reduces the problem of classifying  toric-friendly  groups $G$
to the case where $G$ is semisimple.
The following two lemmas will be used to reduce
the problem of classifying {\em adjoint} semisimple  toric-friendly  groups $G$
to the case where $G$ is an absolutely simple adjoint $k$-group.

\begin{lemma}\label{lem:product}
A direct product $G=G'\times_k G''$ of
connected $k$-groups is  toric-friendly  if and only if
both $G'$ and $G''$ are  toric-friendly.
\end{lemma}

\begin{proof}
Let $K/k$ be a field extension.
Let $T'\subset G'_K$ and  $T''\subset G''_K$ be maximal $K$-tori,
then $T:=T'\times_K T''\subset G_K$ is a maximal $K$-torus,
and every maximal $K$-torus in $G_K$ is of this form.
The commutative diagram
\[
 \xymatrix{
 G(K)  \ar[d] \ar@{=}[r]         & G'(K) \times G''(K)\ar[d]      \\
(G_K/T)(K)   \ar@{=}[r]   &  (G'_K/T')(K) \times (G''_K/T'')(K)
}
\]
shows that every $K$-point of $G_K/T$ lifts to $G$ if and
only if every $K$-point of $G'_K/T'$ lifts to $G'$ and every
$K$-point of $G''_K/T''$ lifts to $G''$.
\end{proof}

\begin{lemma}\label{lem:Weil-restriction}
Let $l/k$ be a finite separable field extension,
$G'$ a connected  $l$-group, and $G=R_{l/k}G'$.
Then $G$ is toric-friendly if and only if $G'$ is toric-friendly.
\end{lemma}

\begin{proof}
Let $K/k$ be a field extension.
Then $l\otimes_k K=L_1\times\dots\times L_r$, where $L_i$ are finite separable extensions of $K$.
It follows that $G_K=\prod_i R_{L_i/K}G'_{L_i}$.
Let $T\subset G_K$ be a maximal $K$-torus, then
$T=\prod_i R_{L_i/K}T'_i$,
where $T'_i$ is a maximal $L_i$-torus of $G'_{L_i}$ for each $i$.
We have
$$
G(K)=G_K(K)=(\prod_i R_{L_i/K}G'_{L_i})(K)=\prod_i G'_{L_i}(L_i)=\prod_i G'(L_i)
$$
and similarly
$
(G_K/T)(K)=\prod_i(G'_{L_i}/T'_i)(L_i),
$
yielding a commutative diagram
$$
\xymatrix{
 G(K)  \ar[d] \ar@{=}[r]         &\prod_i G'(L_i)\ar[d]      \\
(G_K/T)(K)   \ar@{=}[r]   & \prod_i (G'_{L_i}/T'_i)(L_i)
}
$$

If $G'$ is toric-friendly, then the right vertical
arrow in the diagram is surjective, hence the left
vertical arrow is surjective and $G$ is toric-friendly.

Conversely, assume that $G$ is toric-friendly.
Let $L/l$ be a field extension and $T'\subset G'_L$ a maximal $L$-torus.
Set $K :=L$ and $T:=T'$ in the above diagram.
Then we can identify $L$ with one of $L_i$
in the decomposition $l\otimes_k K=L_1\times\dots\times L_r$, say with $L_1$.
In this way we identify $G'_L$ with $G'_{L_1}$  and  $G'_L/T'$
with $G'_{L_1}/T'_1$.
Since $G$ is toric-friendly, the left vertical arrow
in the diagram is surjective,
hence the right vertical arrow is also surjective.
This means that
the map $G'(L_i)\to (G'_{L_i}/T'_i)(L_i)$ is surjective for each $i$
and in particular, for $i = 1$. Consequently,
the map $G'(L)\to (G'_{L}/T')(L)$ is surjective,
and $G'$ is toric-friendly, as desired.
\end{proof}

\section{The elementary obstruction}

\begin{subsec}\label{subsec:ElemObstr}
Let $K$ be a field and $X$ be a smooth geometrically integral $K$-variety.
Write $\ggg=\Gal(\Kbar /K)$, where $\Kbar$ is a fixed separable closure of $K$.
Recall (cf. \cite[Definition 2.2.1]{CTS}),
that the \emph{elementary obstruction}  $\ob(X)$
is the class in $\Ext^1_\ggg(\Kbar (X)^*/\Kbar^*, \Kbar^*)$
of the extension
$$
1\to \Kbar^*\to \Kbar(X)^*\to \Kbar(X)^*/\Kbar ^*\to 1.
$$
In particular, $\ob(X)=0$ if and only if this extension of $\ggg$-modules splits.
Note that if $X$ has a $K$-point, then $\ob(X)=0$, cf. \cite[Proposition 2.2.2(a)]{CTS}.
Conversely, if $Y$ is a $T$-torsor over $K$ for some $K$-torus $T$,
and $\ob(Y)=0$, then $Y$ has a $K$-point, cf. \cite[Lemma 2.1(iv)]{BCS}.
However, if $X$ is an $H$-torsor over $K$
for some simply connected semisimple $K$-group $H$,
then always $\ob(X)=0$, even when $X$ has no $K$-points, see \cite[Lemma 2.2(viii)]{BCS}.
(In \cite{BCS} we always assume that $\charact(K)=0$, but the proofs of
\cite[Lemma 2.2(viii)]{BCS} and \cite[Lemma 2.1(iv)]{BCS}
go through in arbitrary characteristic.)
\end{subsec}

The following key lemma was suggested to us by J.-L. Colliot-Th\'el\`ene.

\begin{lemma}\label{lem:ElemObstr}
Let $K$ be a field, $T$ be a $K$-torus,
$H$ be a  simply connected semisimple $K$-group,
$X$ be a $H$-torsor over $K$ and
$Y$ be a $T$-torsor over $K$.
If $Y$ has an $F$-point over the function field $F = K(X)$ of $X$,
then $Y$ has a $K$-point.
\end{lemma}

\begin{proof}
Since $H$ is simply connected, we have $\ob(X)=0$, see \ref{subsec:ElemObstr} above.
Suppose $Y$ has an $F$-point.
This means that there exist a $K$-rational map $X \dasharrow Y$.
By a lemma of O.~Wittenberg \cite[Lemma 3.1.2]{Wittenberg}, if we have a $K$-rational
map $X\dasharrow Y$ between smooth geometrically
integral $K$-varieties, then $\ob(X)=0$ implies $\ob(Y)=0$.
Since $T$ is a $K$-torus,
if  $\ob(Y)=0$
then $Y(K)\neq \emptyset$, see \ref{subsec:ElemObstr} above.
Thus in our situation $Y$ has a $K$-point, as claimed.
\end{proof}

\begin{lemma} \label{lem:again-CT}
Let $k$ be a field.
Assume that we have a commutative diagram of $k$-groups
$$
\xymatrix{
 S  \ar[d] \ar[r]    &T\ar[d]      \\
H   \ar[r]           &G
}
$$
where  $G$ is a smooth connected  $k$-group,
the vertical map $T\to G$ is the inclusion of
a maximal $k$-torus $T$ into $G$, and
$H$ is semisimple and simply connected.
If there exists a field extension $K/k$
such that the map $H^1(K, S) \to H^1(K, T)$ is non-trivial,
then $G$ is not toric-friendly.
\end{lemma}

\begin{proof} Choose $K$ and $s \in H^1(K, S)$ such
that the image $t\in H^1(K,T)$ of $s$ in $H^1(K,T)$ is non-trivial.
Let $h\in H^1(K,H)$  be the image of $s\in H^1(K, S)$ in $H^1(K,H)$,
and let $g\in H^1(K,G)$ be the image of $t$ (and of $h$) in $H^1(K,G)$,
  as shown in the commutative diagram below:
\[
\xymatrix{
 H^1(K, S)  \ar[d] \ar[r]  &  H^1(K, T) \ar[d]       &  s \ar[r] \ar[d] & t   \ar[d] \\
 H^1(K, H)  \ar[r]         &  H^1(K, G)              &  h \ar[r]        & g          \\
 }
\]
Let $X$ be an $H$-torsor over $K$ representing $h$ and let
$F = K(X)$ be the function field of $X$.
We denote by $h_F$ the image of $h$ in $H^1(F,H)$,
and similarly we define $s_F$, $t_F$, and $g_F$.
Clearly $X$ has an $F$-point, hence $h_F = 1$ in $H^1(F, H)$
and therefore $g_F = 1$ in $H^1(F, G)$.
On the other hand, by Lemma~\ref{lem:ElemObstr}
$t_F \neq 1$. We conclude that the
kernel of the natural map $H^1(F, T) \to H^1(F, G)$ contains
$t_F \neq 1$ and hence, is non-trivial.  This implies
that $G$ is not toric-friendly.
\end{proof}

\begin{subsec}\label{subsec:T-T-sc}
Let $G$ be a reductive $k$-group.
Let $G\sss$ be the derived group of $G$ (it is semisimple),
and let $G\sc$ be the universal cover of $G\sss$
(it is semisimple and simply connected).
Consider the composed homomorphism $f\colon G\sc\onto G\sss\into G$.

Let $K/k$ be a field extension. There is a
a canonical bijective correspondence $T \leftrightarrow T\sc$
between the set of maximal $K$-tori $T\subset G_K$
and the set of maximal $K$-tori $T\sc\subset G\sc$.
Starting from  a maximal $K$-torus $T\subset G_K$,
we define a maximal $K$-torus $T\sc:=f^{-1}(T)\subset G_K\sc$.
Conversely, starting from a maximal $K$-torus  $T\sc\subset G_K\sc$,
we define a maximal $K$-torus $T:=f(T\sc)\cdot R(G)_K\subset G_K$,
where $R(G)$ is the radical of $G$.
\end{subsec}

\begin{proposition} \label{prop.nontrivial}
Let $G$ be a reductive $k$-group.
Let $G\sc$ and $f\colon G\sc\to G$ be as in \ref{subsec:T-T-sc} above.
Let $K/k$ be a field extension,
$T\subset G_K$ be a maximal $K$-torus of $G_K$,
and set $T\sc=f^{-1}(T)\subset G_K\sc$ as above.
If the natural map $H^1(K, T\sc) \to H^1(K, T)$ is non-trivial,
then $G$ is not  toric-friendly.
\end{proposition}

\begin{proof} Immediate from Lemma \ref{lem:again-CT}.
\end{proof}

\begin{proposition} \label{prop.Z}
Let $G$ be a semisimple $k$-group,
$f \colon G\sc \to G$ be the universal covering and
$C := \Ker(f)$.  Then the following conditions are equivalent:

\smallskip
(a) $G$ is  toric-friendly.

\smallskip
(b) The map $H^1(K, T\sc) \to H^1(K, T)$ is
trivial (i.e., is identically zero) for every
field extension $K/k$ and every maximal $K$-torus $T\sc$ of $G\sc$.
Here $T := f(T\sc)$.

\smallskip
(c) The map $H^1(K, C) \to H^1(K, T\sc)$ is
surjective for every
field extension $K/k$ and every maximal $K$-torus
$T\sc$ of $G\sc$.

\smallskip
(d) The connecting homomorphism
$\partial_T \colon H^1(K, T) \to H^2(K, C)$ is
injective for every
field extension $K/k$ and every maximal $K$-torus
$T$ of $G$.

\smallskip
(e) The natural map $H^1(K, T) \to H^1(K, G)$ is
injective for every
field extension $K/k$ and every maximal $K$-torus
$T$ of $G$.
\end{proposition}

\begin{proof} (a) $\Longrightarrow$ (b) by
Proposition~\ref{prop.nontrivial}.
Examining the cohomology sequence
\[ H^1(K, C) \to H^1(K, T\sc) \to H^1(K, T) \to H^2(K, C) \]
associated to the exact sequence $ 1 \to C \to T\sc \to T \to 1$
of $k$-groups, we see that (b), (c) and (d) are equivalent.

\smallskip
(d) $\Longrightarrow$ (e): The diagram
\[ \xymatrix{
1 \ar[r] & C \ar[r] \ar@{=}[d] & T\sc \ar[r] \ar@{_{(}->}[d] & T
\ar[r]  \ar@{_{(}->}[d] & 1 \\
1 \ar[r] &  C \ar[r] & G\sc \ar[r] & G  \ar[r] & 1} \]
of $K$-groups induces compatible connecting morphisms
\[ \xymatrix{
 H^1(K, T)  \ar[dd] \ar[rd]^{\partial_T}   &     \\
      & H^2(K, C)   \\
 H^1(K, G)  \ar[ru]^{\partial_G} &   } \]
Suppose $\alpha, \beta \in H^1(K, T)$ map to the same element
in $H^1(K, G)$. Then the above diagram shows that
$\partial_T(\alpha)= \partial_T(\beta)$ in $H^2(K, C)$.
Part (d) now tells us that $\alpha = \beta$.

\smallskip
(e) $\Longrightarrow$ (a) is obvious, since (a) is equivalent to the assertion that
the map $H^1(K, T) \to H^1(K, G)$ has trivial kernel for every
$K$ and $T$, see Definition~\ref{def:toric-friendly}.
\end{proof}

\begin{corollary} \label{cor.Z}
With the assumptions and notation of Proposition~\ref{prop.Z},
if $G$ is  toric-friendly  and quasi-split, then

\smallskip
(a) the map $H^1(K, G\sc) \to H^1(K, G)$ is
trivial for every $K/k$,

\smallskip
(b) the map $H^1(K, C) \to H^1(K, G\sc)$ is
surjective for every $K/k$,

\smallskip
(c) the connecting map
$\partial_G \colon H^1(K, G) \to H^2(K, C)$ has trivial
kernel for every $K/k$.
\end{corollary}

\begin{proof}
Examining the cohomology sequence
\[ H^1(K, C) \to H^1(K, G\sc) \to H^1(K, G) \to H^2(K, C) \]
associated to the exact sequence $1 \to C \to G\sc \to G \to 1$,
we see that (a), (b) and (c) are equivalent.

To prove (a), recall that since $G_K$ is quasi-split,
by a theorem of Steinberg~\cite[Theorem 1.8]{steinberg}
every $x\sc \in H^1(K, G\sc)$ lies in the image of the map
$H^1(K, T\sc) \to H^1(K, G\sc)$ for some maximal $K$-torus
$T\sc$ of $G_K\sc$. Since $G$ is  toric-friendly, by Proposition~\ref{prop.Z} the map
$H^1(K, T\sc)  \to  H^1(K, T)$
is trivial.  The commutative diagram
\[ \xymatrix{
 H^1(K, T\sc)  \ar[d] \ar[r]    &  H^1(K, T) \ar[d]      \\
 H^1(K, G\sc)  \ar[r]          &  H^1(K, G) } 
\]
now shows that the image of $x\sc$ in $H^1(K,G)$ is 1.
Thus the map
$H^1(K, G\sc)  \to  H^1(K, G)$ is trivial.
\end{proof}

\begin{theorem} \label{thm.special}
Let $G$ be a split semisimple $k$-group
and $f \colon G\sc \to G$ be its universal covering map.
If $G$ is  toric-friendly  then $G\sc$ is special.
\end{theorem}

\begin{proof}
Let $T\sc$ be a split maximal torus of $G\sc$. Recall that $T\sc$
is special (as is any split torus).
Set $C=\ker f$,  then $C \subset T\sc$.
For any field  extension $K/k$,
the map $H^1(K, C) \to H^1(K, G\sc)$ factors through
$H^1(K, T\sc) =  1 $ and hence is trivial.
By Corollary~\ref{cor.Z}(b) this map is also surjective.
This shows that $H^1(K, G\sc) =  1 $ for every $K/k$,
that is, $G\sc$ is special.
\end{proof}

\begin{remark} \label{rem.quasisplit}
Our proof of Theorem~\ref{thm.special} goes through for any
(not necessarily split) semisimple $k$-group $G$, as long
as $G\sc$ contains a special maximal $k$-torus $T\sc$.
In particular,
Theorem~\ref{thm.special} remains valid
for any quasi-split semisimple $k$-group $G$, in view of
Lemma~\ref{lem:quasi-split-quasi-trivial} below.
This lemma is a special case of~\cite[Lemma 5.6]{cgp};
however, for the sake of completeness we supply
a short self-contained proof.
\end{remark}

\begin{lemma} \label{lem:quasi-split-quasi-trivial}
Let $G$ be a semisimple, simply connected,
quasi-split $k$-group over a field $k$.
Let $B\subset G$ be a Borel subgroup defined over $k$,
and let $T\subset B\subset G$ be a maximal $k$-torus of $G$ contained in $B$.
Then $T$ is a quasi-trivial $k$-torus.
\end{lemma}

\begin{proof}
We write $\kbar$ for a fixed algebraic closure of $k$.
Let $\X^\vee(T)$ denote the group of cocharacters of $T$.
Let $R^\vee=R^\vee(G_\kbar,T_\kbar)\subset \X^\vee(T)$
denote the coroot system of $G_\kbar$ with respect to $T_\kbar$,
and let $\Pi^\vee\subset R^\vee$ denote
the basis of $R^\vee$ corresponding to $B$.
The Galois group $\Gal(k_s/k)$ acts on $\X^\vee(T)$.
Since $T$, $G$, and $B$ are defined over $k$,
the subsets $R^\vee$ and $\Pi^\vee$ of $\X^\vee(T)$
are invariant under this action.  Since $G$ is simply connected,
$\Pi^\vee$ is a $\bbZ$-basis of $\X^\vee(T)$.
Thus $\Gal(k_s/k)$ permutes the $\bbZ$-basis $\Pi^\vee$ of $\X^\vee(T)$;
in other words, $T$ is a quasi-trivial torus.
\end{proof}

\begin{remark}
 A similar assertion for  \emph{adjoint} quasi-split groups was proved by G. Prasad \cite[Proof of Lemma 2.0]{Prasad}.
\end{remark}

\section{Examples in type $A$}

Let $k$ be a field and $A$ a central simple $k$-algebra of dimension $n^2$.
We write $\GL_{1,A}$ for the $k$-group with $\GL_{1,A}(R)=(A\otimes_k R)^*$
for any unital commutative $k$-algebra $R$
(here $(\ )^*$ denotes the group of invertible elements).
The $k$-group $\GL_{1,A}$ is an inner form of $\GL_{n,k}$.

Let $K$ be a field.
Recall that an $n$-dimensional commutative \'etale $K$-algebra is a finite product
$E=\prod_i L_i$, where each $L_i$ is a finite separable field extension of $K$ and $\sum_i[L_i:K]=n$.
For such $E=\prod_i L_i$ we define a $K$-torus
$R_{E/K}\G_{m,E}:=\prod_i R_{L_i/K}\G_{m,L_i}$,  then $(R_{E/K}\G_{m,E})(K)=E^*$.
Clearly the $K$-torus $R_{E/K}\G_{m,E}$ is  quasi-trivial.

\begin{proposition}\label{lem:GL(A)}
Let $k$ be a field, and let  $A/k$ be a central simple $k$-algebra
of dimension $n^2$.

(a) The $k$-group $G=\GL_{1,A}$ is  toric-friendly.

(b) The $k$-group $\PGL_{1,A}:=\GL_{1,A}/\G_{m,k}$ is  toric-friendly.

(c) In particular, $\GL_{n,k}$ and $\PGL_{n,k}$ are  toric-friendly.
\end{proposition}

\begin{proof}
(a) Let $K/k$ be a field extension and  $T\subset G_K=\GL_{1,A\otimes_k K}$ be a maximal $K$-torus.
Let $E$ be the centralizer of $T$ in $A\otimes_k K$.
An easy calculation over a separable closure $K_s$ of $K$ shows that $E$ is an $n$-dimensional
commutative \'etale $K$-subalgebra of $A\otimes_k K$ and that $T=R_{E/K}\G_{m,E}$.
It follows that $T$ is quasi-trivial, hence special.
Since all  maximal $K$-tori $T\subset G_K$ are special, $G$ is toric-friendly.

(b) follows from (a) and Corollary~\ref{cor.modR(G)}.
To deduce (c) from (a) and (b), set $A = M_{n}(k)$ (the matrix algebra).
\end{proof}

We now come to the main result of this section which asserts that
a toric-friendly semisimple groups of type $A$ is necessarily an adjoint group.

\begin{proposition} \label{prop.typeA}
Let $k$ be a field.
Consider a $k$-group
$G = (\SL_{n_1} \times \dots \times \SL_{n_r})/C$, where
$C \subset \mu : = \mu_{n_1} \times \dots \times \mu_{n_r}$ is a central
subgroup of $G\sc = \SL_{n_1} \times \dots \times \SL_{n_r}$
(not necessarily smooth).
If $C \neq \mu$ then  $G$ is not toric-friendly.
\end{proposition}

Before proceeding with the proof, we  fix some notation.
Let $L/K$ be a finite separable field extension of degree $n$.  Set
$$
R^1_{L/K}(\G_m) := \ker[N_{L/K}\colon R_{L/K}\G_{m,L}\to \G_{m,K}],
$$
where $N_{L/K}$ is the norm map. Clearly
$R^1_{L/K}(\G_m)$ can be embedded into  $\SL_{n,K}$ as a  maximal $K$-torus.
The embedding $K\into L$ induces
an embedding $\mu_{n,K}\into R^1_{L/K}\Gm$, where $n=[L:K]$.

The following two lemmas are undoubtedly known. We include  short
proofs below because we have not been able to find appropriate references.

\begin{lemma}\label{lem:L}
There is a commutative diagram
\begin{equation}\label{eq:L}
\xymatrix{
K^*/K^{* n} \ar[r]^{\simeq} \ar[d] &H^1(K,\mu_n) \ar[d] \\
K^*/N_{L/K}(L^*) \ar[r]^{\simeq}  &H^1(K,R^1_{L/K}\Gm)}
\end{equation}
where the horizontal arrows are canonical isomorphisms,
the right vertical arrow is induced by the embedding $\mu_n\into R^1_{L/K}\Gm$,
and the left vertical arrow is the natural projection.
\end{lemma}

\def\id{{\rm id}}
\begin{proof}
Apply the flat cohomology functor to the commutative
diagram of commutative $K$-groups
$$
\xymatrix{
1\ar[r] &\mu_{n,K} \ar[r]\ar@{_{(}->}[d] &\G_{m,K} \ar[r]^n \ar@{_{(}->}[d]  &\G
_{m,K} \ar[r]\ar[d]^\id &1\\
1\ar[r] &R^1_{L/K}\Gm \ar[r]             &R_{L/K}\Gm  \ar[r]^{N_{L/K}}       &\G
_{m,K} \ar[r]           &1
}
$$
and use Hilbert's theorem 90.
\end{proof}

\begin{lemma}\label{lem:mu}
Suppose $r\mid n$. Then there is a commutative diagram
\begin{equation*}
\xymatrix{
K^*/K^{* n} \ar[r]^{\simeq} \ar[d]  &H^1(K,\mu_n) \ar[d]^{(n/r)_*} \\
K^*/K^{* r} \ar[r]^{\simeq}         &H^1(K,\mu_r) \, ,
}
\end{equation*}
where the horizontal arrows are canonical isomorphisms,
the right vertical arrow is induced by
the homomorphism $\mu_n \labelto{n/r} \mu_r$
given by $x\mapsto x^{n/r}$,
and the left vertical arrow is the natural projection.
\end{lemma}

\begin{proof}
Similar to that of Lemma \ref{lem:L} using the commutative diagram
$$
\xymatrix{
1\ar[r] &\mu_{n} \ar[r]\ar[d]^{n/r} &\Gm \ar[r]^n \ar[d]^{n/r}  &\Gm \ar[r]\ar[d]^\id  &1\\
1\ar[r] &\mu_{r} \ar[r]       &\Gm \ar[r]^r                     &\Gm \ar[r]            &1
}
$$
\end{proof}

\begin{example}\label{ex:SLn}
The group $G=\SL_{n,k}$ ($n\ge 2$) is not toric-friendly.
\end{example}

\begin{proof}
Since $\SL_n$ is special, it suffices to construct an extension $K/k$
and a maximal $K$-torus $T := R^1_{L/K}(\Gm)$ such that $H^1(K,T)\neq 1$.
In view of Lemma \ref{lem:L} it suffices to show that
$N_{L/K}(L^*) \neq K^*$ for some field extension $K/k$ and
some  finite separable field extension $L/K$ of degree $n$.
This is well known, see e.g. the proof of \cite[Proposition 3.1.46]{rowen}.
We include a short proof below as a way of motivating
a related but more complicated argument  at the end
of the proof of Proposition~\ref{prop.typeA}.

Let $L := k(x_1,\dots, x_n)$,
where $x_1, \dots, x_n$ are independent variables, and
$K := L^\Gamma$, where $\Gamma$ is the cyclic
group of order $n$ which acts on $L$ by cyclically permuting
$x_1, \dots, x_n$.  For $0 \ne a \in k[x_1,\dots,x_n]$,
let $\deg(a) \in \N$ denote the degree of $a$ as a polynomial
in $x_1, \dots, x_n$.  If  $a\in k(x_1, \dots, x_n)$,
$a = \frac{b}{c}$ with $ 0 \ne b, c \in k[x_1, \dots, x_n]$,
then we define $\deg(a) = \deg(b) - \deg(c)$.
This yields the usual degree homomorphism $\deg\colon L^* \to \Z$.
Since $\Nm_{L/K}(a)=\prod_{\gamma \in \Gamma} \gamma(a)$,
we see that $\deg(\Nm_{L/K}(a))=n\deg(a)$ is divisible by $n$,
for every $a \in L^*$.
On the other hand, $s_1=x_1+\dots+x_n \in K$ has degree $1$.
This shows that $N_{L/K}(L^*) \neq K^*$, as claimed.
\end{proof}

\begin{subsec}
{\em Proof of Proposition~\ref{prop.typeA}.}
Let $K/k$ be a field extension.
For each $i=1,\dots,r$,
let $L_i$ be a separable field extension of degree $n_i$ over $K$,
and  let $T = T_1 \times \dots \times T_r$ be a maximal $K$-torus of
$G\sc$, where  $T_i := R_{L_i/K}^1(\bbG_m)$.
By Proposition~\ref{prop.Z} it suffices to show that the composition
\begin{equation} \label{e.projection2}
H^1(K, C) \to H^1(K, \mu) \to H^1(K, T)
\end{equation}
is not surjective for some choice of extensions $K/k$ and $L_i/K_i$.
Since $C \subsetneqq \mu$, there exist
a prime $p$ and a non-trivial character $\chi \colon
\mu \to \mu_p$ such that $\chi(C) = 1$.
By Proposition \ref{prop:CGG'}(a) we may
assume that $C = \Ker(\chi)$.
For notational simplicity, let us suppose that $n_1, \dots, n_s$
are divisible by $p$ and $n_{s+1}, \dots, n_r$ are not, for
some $0 \le s \le r$.
Then it is easy to see that $\chi$ is of the form
\[
\chi(c_1, \dots, c_r) =
c_1^{d_1 \frac{n_1}{p}} \cdots c_s^{d_s \frac{n_s}{p}}
\]
for some integers $d_1, \dots, d_s$. Since $\chi$
is non-trivial on $\mu$, we have
$s \ge 1$ and $d_i$ is not divisible by $p$ for
some $i = 1, \dots, s$, say for $i = 1$.
That is, we may assume that $d_1$ is not divisible by $p$.

Lemma \ref{lem:L} gives
a concrete description of the second map in~\eqref{e.projection2}.
To determine the image of the map $H^1(K, C) \to H^1(K, \mu)$,
we examine the cohomology exact sequence
\[ \xymatrix{
 H^1(K, C)  \ar[r] & H^1(K, \mu) \ar@{=}[d] \ar[r]^{\chi_*}
& H^1(K, \mu_p) \ar@{=}[d] \\
 & \prod_{i = 1}^r K^*/K^{* n_i} \ar[r]^{\quad \chi_*} & K/K^{* p}}
\]
induced by the exact sequence $1 \to C \to \mu \xrightarrow{\chi} \mu_p \to 1$.
The image of $H^1(K, C)$ in $H^1(K, \mu)$ is the kernel of $\chi_*$.
By Lemma \ref{lem:mu} $\chi_*$ maps
the class of $(a_1, \dots, a_r)$ in $H^1(K, \mu) =  \prod_{i = 1}^r
K^*/K^{* n_i}$
to the class of $a_1^{d_1} \cdots a_s^{d_s}$ in $H^1(K, \mu_p)=K/K^{* p}$.
In other words,
the image of $H^1(K, C)$ in $H^1(K, \mu)$ is the subgroup of classes of
$r$-tuples $(a_1, \dots, a_r)$ in $H^1(K, \mu) =  \prod_{i = 1}^r
K^*/K^{* n_i}$ such that $a_1^{d_1} \dots a_s^{d_s} \in K^{* p}$.
Consequently, the image of $H^1(K, C)$ in
$H^1(K, T) =  \prod_{i = 1}^r K^*/N_{L_i/K}(L_i^*)$
consists of  classes of $r$-tuples $(a_1, \dots, a_r)$ such that
$a_1^{d_1} \dots a_s^{d_s} \in K^{* p}$.

It remains to construct a field extension $K/k$, separable field
extensions $L_i/K$ of degree $n_i$ for $i = 1, \dots, r$,
and an element $\alpha \in H^1(K, T) =
\prod_{i = 1}^r K^*/N_{L_i/K}(L_i^*)$ which cannot be represented by
$(a_1, \dots, a_r) \in (K^*)^r$
such that $a_1^{d_1} \cdots a_s^{d_s} \in K^{* p}$.
This will show that the map $H^1(K, C) \to H^1(K, T)$
is not surjective, as claimed.

Set $L := k(x_1, \dots, x_n)$, where $n = n_1 + \dots + n_r$
and $x_1, \dots, x_n$ are independent variables.
The symmetric group $S_n$ acts on $L$ by permuting these variables;
we embed $S_{n_1} \times \dots \times S_{n_r}$ into $S_n$ in the natural
way, by letting $S_{n_1}$ permute the first $n_1$ variables,
$S_{n_2}$ permute the next $n_2$ variables, etc.
Set $K := L^{S_{n_1} \times \dots \times S_{n_r}}$,
$s_1 := x_1 + \dots + x_n \in K$ and
\[ L_1 := K(x_1), \; \;  L_2 := K(x_{n_1 + 1}), \; \;  \dots \; \;
L_r := K(x_{n_1 + \dots + n_{r-1} + 1}) \, . \]
Clearly $[L_i:K] = n_i$.
We claim that the class
of $(s_1, 1, \dots, 1)$ in
$\prod_{i = 1}^r K^*/N_{L_i/K}(L_i^*)$ cannot be represented by
any $(a_1, \dots, a_r) \in (K^*)^r$ with
$a_1^{d_1} \cdots a_s^{d_s} \in K^{* p}$.

Let $\deg \colon L^* \to \Z$ be the degree map, as
in Example~\ref{ex:SLn}. Arguing as we did there, we see that
$\deg(N_{L_i/K}(a))$ is divisible by $n_i$ for every
$i = 1, \dots, r$ and every $a \in L_i^*$.  In particular,
$(a_1, \dots, a_r) \mapsto \deg(a_i)+n_i\Z$
is a well-defined function
$\prod_{i = 1}^r K^*/N_{L_i/K}(L_i^*) \to \bbZ/ n_i \bbZ$,
and consequently,
\[ f(a_1, \dots, a_n) := d_1 \deg(a_1) + \dots + d_s \deg(a_s)+p\Z \]
is a well-defined function $H^1(K, T) \to \bbZ/ p \bbZ$.
We have $f(a_1,\dots,a_n)=\deg(a_1^{d_1} \cdots a_s^{d_s})$.
If $a_1^{d_1} \cdots a_s^{d_s} \in K^{* p}$ then
$f(a_1, \dots, a_r)= 0$ in $\bbZ/p \bbZ$.
On the other hand, since $\deg(1) = 0$, $\deg(s_1) = 1$
and $d_1$ is not divisible by $p$, we conclude
that $f(s_1, 1, \dots, 1) \neq 0$ in $\bbZ/p \bbZ$.
This proves the claim and the proposition.
\qed
\end{subsec}

\section{Groups of type $C_n$ and outer forms of $A_n$}

\begin{proposition} \label{prop:C}
 No absolutely simple $k$-group of type $C_n$ ($n \ge 2$)
is  toric-friendly.
\end{proposition}

\begin{proof}
 Clearly we may assume that $k$ is algebraically closed.
We may assume also that $G$ is adjoint,
see Proposition~\ref{prop:CGG'}(a).
We see that $G=\PSp_{2n}$ and $G\sc =\Sp_{2n}$.
By Example~\ref{ex:SLn}
$\SL_2$ is not toric-friendly.
This means that there exist a field extension $K/k$, a maximal $K$-torus
$S\subset \SL_{2,K}$, and a cohomology class $a_S\in H^1(K,S)$ such that $a_S\neq 1$.
We consider the standard embedding
$$
(\SL_2)^n=(\Sp_2)^n\into \Sp_{2n},\quad n\ge 2.
$$
Set $T\sc=S^n\subset (\Sp_2)^n\subset\Sp_{2n}=G\sc$.
Let $\iota\colon S\into T\sc=S^n$ be the embedding as the first factor.
Set $a\sc=\iota_*(a_S)\in H^1(K,T\sc)$.
Let $T$ be the image of $T\sc$ in $G=\PSp_{2n}$, and let
$a$ be the image of $a\sc$ in $H^1(K,T)$.

Now observe that the homomorphism
$$
\chi\colon T\sc=S^n\to S,\quad (x_1,\dots,x_n) \mapsto x_1 x_2^{-1}
$$
factors through $T$ (recall that  $n\ge 2$).
Since $\chi\circ \iota={\rm id}_S$, we see that $a\neq 1$.
On the other hand, the image of $a\sc$ in $H^1(K,G\sc)$ is 1 (because $G\sc=\Sp_{2n}$ is special),
hence $a\in \ker[H^1(K,T)\to H^1(K,G)]$, and we see that
$G=\PSp_{2n}$ is not toric-friendly.
\end{proof}

\begin{proposition}\label{prop:outer-A}
 No absolutely simple $k$-group of outer type $A_n$ ($n \ge 2$)
is toric-friendly.
\end{proposition}

Our proof of Proposition~\ref{prop:outer-A} will rely on the following lemma.

\begin{lemma}\label{lem: division-algebras}
Let $k$ be a field, $K/k$ a separable quadratic extension,
and $D/K$ a central division algebra of dimension $r^2$ over $K$
with an involution $\sigma$ of the second kind
(i.e., $\sigma$ acts non-trivially on $K$ and trivially on $k$).
Then there exists a finite separable field extension $F/k$ such
that $L:=K\otimes_k F$ is a field, and $D\otimes_K L$ is split
(i.e., is $L$-isomorphic to the matrix algebra $M_r(L)$).
\end{lemma}

\begin{proof}[Proof of the lemma]
Since there are no non-trivial central division algebras over finite fields,
we may assume that $k$ and $K$ are infinite.
Let \[ H=\{x\in D\ |\ x^\sigma=x\} \]
denote the $k$-space of Hermitian elements of $D$.
Consider the embedding $D\hookrightarrow M_r(\Kbar )$ induced
by an isomorphism $D\otimes_K \Kbar \simeq M_r(\Kbar)$,
where $\Kbar $ is a separable closure of $K$.
An element $x$ of $D$ is called semisimple regular
if its image in $D\otimes_K \Kbar \simeq M_r(\Kbar )$
is a semisimple matrix with $r$ distinct eigenvalues.
A standard argument using  an isomorphism
$D\otimes_k  \Kbar \simeq M_r(\Kbar )\times M_r(\Kbar )$
shows that there is a dense open subvariety
$H_{\reg}$  in the space $H$, consisting of semisimple regular elements.
Clearly $H_{\reg}$ is defined over $k$ and contains $k$-points.

Let $x\in H_{\reg}(k)\subset D$ be a semisimple regular Hermitian element.
Let $L$ be the centralizer of $x$ in $D$.
Since $x$ is Hermitian ($\sigma$-invariant),
the $k$-algebra $L$ is $\sigma$-invariant.
Since $x$  is semisimple and regular, the algebra $L$ is
a commutative \'etale $K$-subalgebra of $D$
of dimension $r$ over $K$ (we calculate in $D\otimes_K  \Kbar $).
Clearly $L$ is a field, $[L:K]=r$, and $L$ is separable over $k$.
Since $L\subset D$ and $[L:K]=r$, the field $L$ is a splitting field for $D$,
see e.g. \cite[Corollary 13.3]{Pi}.

Since $L\supset K$, we see that $\sigma$ acts non-trivially on $L$.
Let $F = L^{\langle \sigma \rangle}$ denote
the subfield of $L$ consisting of elements
fixed by $\sigma$. Then $[L:F]=2$ and $[F:k]=r$.
Clearly $F$ is separable over $k$.
We have  $F\cap K=k$ and $FK=L$, hence $L=K\otimes_k F$.
\end{proof}

\begin{subsec}{\em Proof of Proposition \ref{prop:outer-A}.}
By Proposition~\ref{prop:CGG'}(a) we may assume that $G$ is adjoint.
By Lemma~\ref{lem: division-algebras}
there is a finite separable field extension $F/k$
such that $G_F\simeq \PSU(L^{n+1}, h)$, where $L/F$
is a separable quadratic extension
and $h$ is a Hermitian form on $L^{n+1}$.
It suffices to prove that $G_F=\PSU(L^{n+1},h)$ is not  toric-friendly.

Set $S=R^1_{L/F}\Gm$.
We set $G_F\sc=\SU(L^{n+1}, h)$.
We may assume that $h$ is a diagonal form,
see \cite[Proposition (6.2.4)(1)]{Knus} or \cite[Theorem 7.6.3]{Scharlau}.
Consider the diagonal torus $S^{n+1}\subset {\rm U}(L^{n+1},h)$ and set
$T\sc=S^{n+1}\cap\SU(L^{n+1},h)$.

We claim that there exists a field extension $K/F$ such
that $H^1(K,S)\neq 1$. Indeed, take
$K=F((t))$, the field of formal Laurent series over $F$.
Then by \cite[Prop.~V.2.3(c)]{Serre-CL}
$H^1(K,S)\simeq H^1(F,S)\times \Z/2\Z\neq  1 $.

Now let $a_S\in H^1(K,S)$, $a_S\neq 1$, and consider the embedding
$$
\iota\colon S\into T\sc\subset S^{n+1},\quad x\mapsto (x, x^{-1},1,\dots,1).
$$
Set $a\sc=\iota_*(a_S)\in H^1(K,T\sc)$.
Let $T$ be the image of $T\sc$ in $G_F= \PSU(L^{n+1}, h)$ and
$a$ be the image of $a\sc$ in $H^1(K,T)$.

Note that the homomorphism
$$
\chi\colon T\sc\to S,\quad (x_1,\dots,x_n,x_{n+1}) \mapsto x_1 x_3^{-1}
$$
factors through $T$  (recall that  $n\ge 2$).
Since $\chi\circ \iota={\rm id}_S$, we see that $a\neq 1$.
Now by Proposition \ref{prop.nontrivial} $G_F$ and hence $G$ are not toric-friendly.
\qed
\end{subsec}

\section{Classification of semisimple  toric-friendly  groups}

\begin{lemma} \label{lem:alg-closed}
Let $k$ be an algebraically closed field.
If a  semisimple $k$-group $G$ is  toric-friendly,
then it is adjoint of type $A$,
that is, $G \simeq\prod_i \PGL_{n_i}$ for some integers $n_i \ge 2$.
\end{lemma}

\begin{proof} First assume that $G$ is simple.
By Theorem ~\ref{thm.special} the simply connected cover
$G\sc$ of $G$ is special. By a theorem
of Grothendieck~\cite[Theorem 3]{Grothendieck}
$G\sc$ is special if and only if $G$ is of type $A_n$, $n \ge 1$ or
$C_n$, $n \ge 2$.  Proposition~\ref{prop:C} rules out
the second possibility. Thus $G$ is of type $A$.

Now let $G$ be semisimple.
By Proposition \ref{prop:CGG'}(a) $G\ad$ is  toric-friendly.
Write $G\ad=\prod_i G_i$, where each $G_i$ is an adjoint
simple group, then by Lemma \ref{lem:product} each $G_i$ is  toric-friendly.
As we have seen, this implies that each
$G_i$ is of type $A$, i.e., is isomorphic to $\PGL_{n_i}$ for some $n_i$.
By Proposition \ref{prop.typeA} $G$ is adjoint, that is,
$G=G\ad=\prod_i \PGL_{n_i}$.
\end{proof}

\begin{subsec}{\em Proof of Main Theorem \ref{thm:non-split-intro}.}
If $G$ is toric-friendly, then clearly $G_\kbar$ is toric-friendly,
where $\kbar$ is an algebraic closure of $k$.
By Lemma \ref{lem:alg-closed} $G$ is adjoint of type $A$.
Write $G=\prod_i R_{F_i/k}G'_i$, where each $F_i/k$ is a finite separable extension
and $G'_i$ is a form of $\PGL_{n_i, F_i}$.
By Lemmas \ref{lem:product} and \ref{lem:Weil-restriction} each $G'_i$ is toric-friendly,
and by Proposition~\ref{prop:outer-A} $G'_i$ is an \emph{inner} form of $\PGL_{n_i, F_i}$.

Conversely, by Proposition~\ref{lem:GL(A)}
an inner form $G'_i$ of $\PGL_{n_i, F_i}$ is toric-friendly.
By Lemmas~\ref{lem:Weil-restriction} and~\ref{lem:product}
the product $G=\prod_i R_{F_i/k}G'_i$ is toric-friendly.
\qed
\end{subsec}

\begin{corollary} \label{cor.non-special}
Let $G$ be a nontrivial semisimple $k$-group.
Then there exist a field extension $K/k$ and
a maximal $K$-torus $T \subset G$ which is not special.
Equivalently, there exists a field extension $K/k$ and
a maximal $K$-torus $T$ of $G$ such that $H^1(K, T) \ne  1 $.
\end{corollary}

\begin{proof} Assume the contrary, that is,
that for any field extension $K/k$,
any maximal $K$-torus $T \subset G_K$ is special.
We may and shall assume that $G$ is split.
Recall that for a (quasi-)split group,
by a theorem of Steinberg~\cite[Theorem 11.1]{steinberg}
every element of $H^1(K, G)$ lies in the image of the map
$H^1(K, T) \to H^1(K, G)$ for some maximal $K$-torus
$T$ of $G$. Thus, under our assumption we have
$H^1(K, G) =  1 $ for every field extension $K/k$, that is, $G$ is special.
By a theorem of Grothendieck~\cite[Theorem 3]{Grothendieck} this is only
possible if $G$ is simply connected and has components only of types $A$ and $C$.
On the other hand, $G$ is clearly  toric-friendly
(see Definition~\ref{def:toric-friendly}),
and by Theorem~\ref{thm:non-split-intro}
no nontrivial simply connected semisimple group can be  toric-friendly, a contradiction.
\end{proof}

The following corollary follows immediately
from Theorem \ref{thm:non-split-intro} and Corollary \ref{cor.modR(G)}.

\begin{corollary}\label{cor:reductive-tf}
Let $G$ be a split reductive  $k$-group.
The group $G$ is  toric-friendly  if and only if it satisfies the following two
conditions:
\par (a) the center $Z(G)$ of $G$ is a $k$-torus, and
\par (b) the adjoint group $G\ad:=G/Z(G)$ is a direct product of
simple adjoint groups of type $A$. \qed
\end{corollary}

Note that in condition (a) we allow the trivial $k$-torus $\{ 1 \}$.
\medskip

By Corollary \ref{cor.modR(G)} if $G$ is a reductive  $k$-group
such that   $G/R(G)$ is  toric-friendly  and $R(G)$ is special, then $G$ is  toric-friendly.
The example below shows that when $G/R(G)$ is  toric-friendly  but  $R(G)$ is not special,
$G$ need not be  toric-friendly.

\def\U{{\rm U}}
\begin{example}\label{ex:non-special}
Let $k=\R$, $G=\U_2$,  the unitary group in 2 complex variables.
Then $Z(G)$ is the group of scalar matrices in $G$, it is connected,
hence $R(G)=Z(G)$ and $G/R(G)=G\ad=\PSU_2$.
Since $\PSU_2$ is an inner form of $\PGL_{2,\R}$,
by Theorem \ref{thm:non-split-intro} it is  toric-friendly.
However, the group $G=\U_2$ is not  toric-friendly.
This does not contradict to Corollary \ref{cor.modR(G)},
because $R(G)=Z(G)$ is not special: $H^1(\R,Z(G))=\R^*/N_{\C/\R}(\C^*)\simeq\Z/2\Z$.
\end{example}

\begin{proof}
We prove that $G=\U_2$ is not  toric-friendly.
Set $S=R^1_{\C/\R}\Gm$.
Let $T$ be the diagonal maximal $\R$-torus of $\U_2$.
Set $G\sc=\SU_2$, $T\sc=T\cap \SU_2$, then $T\sc\simeq S$.

Let $a\sc\in H^1(\R,T\sc)$ be the cohomology class of the cocycle
given by the element  $-1\in T\sc(\R)$ of order 2.
Let $a\in H^1(\R,T)$ be  the image of $a\sc$ in $H^1(\R,T)$.
Clearly  $a\neq 1$.
By Proposition \ref{prop.nontrivial} $G$ is not toric-friendly.
\end{proof}

\end{document}